\documentclass[english,a4paper,12pt]{amsart}
\usepackage{graphicx}
\usepackage{combelow}
\usepackage[all,web]{xy}
\usepackage{amssymb}

\usepackage{amsmath,amsthm}
\newtheorem{theorem}{Theorem}
\newtheorem{proposition}{Proposition}
\newtheorem{remark}{Remark}
\newtheorem{lemma}{Lemma}
\newtheorem{conjecture}{Conjecture}

\def\CC{\mathbb{C}}

\def\PP{\mathbb{P}}
\def\QQ{\mathbb{Q}}
\def\RR{\mathbb{R}}

\def\ZZ{\mathbb{Z}}
\newcommand{\mtwo}[4]{\left(
        \begin{matrix}#1&#2\\#3&#4
        \end{matrix}\right)}
	
\begin{document}

\markboth{C. Casta\~no-Bernard \& F. Luca}{On Petersson's partition limit formula}

\title{On Petersson's partition limit formula}

\author{Carlos Casta\~no-Bernard}
\address{CUICBAS, Universidad de Colima, Address\\
Colima, Mexico}
\email{ccastanobernard@gmail.com}

\author{Florian Luca}
\address{School Of Mathematics, Wits University, Address\\
Johannesburg, South Africa\\
Research Group in Algebraic Structures and Applications\\
King Abdulaziz University, Address\\
Jeddah, Saudi Arabia\\
Centro de Ciencias Matem\'aticas, UNAM, Address\\
Morelia, Mexico}
\email{florian.luca@wits.ac.za}

\maketitle

\begin{abstract}
For each prime $p\equiv 1\pmod{4}$ consider the Legendre character $\chi=(\frac{\cdot}{p})$.
Let $p_\pm(n)$
be the number of partitions of $n$ into parts $\lambda>0$
such that $\chi(\lambda)=\pm 1$.
Petersson proved a beautiful limit formula for the ratio of $p_+(n)$ to $p_-(n)$
as $n\to\infty$
expressed in terms of important invariants of the real quadratic field $K=\QQ(\sqrt{p})$.
But his proof is not illuminating and Grosswald conjectured a more natural proof
using a Tauberian converse of the Stolz-Ces\`aro theorem.
In this paper we suggest an approach to address Grosswald's conjecture.
We discuss a monotonicity conjecture which looks quite natural
in the context of the monotonicity theorems of Bateman-Erd\H{o}s.
\end{abstract}

\section{Introduction}

Let $K$ be a real quadratic field, $h_K$ its class number,
and $\varepsilon_K>1$ its fundamental unit.
Let us assume that the discriminant of $K$ is a prime number $p$,
so in particular $p\equiv 1\pmod{4}$.
Consider the Nebentypus cover $X_\chi(p)$
of degree two of the modular curve $X_0(p)$ introduced by Shimura~\cite[p. 174]{shimura:lnm}
in his work towards a theory of ``real multiplication''.\footnote{
Shimura determines division points of certain one-dimensional factors
of the Jacobian $J_\chi(p)$ of $X_\chi(p)$ that generate abelian extensions of $K$.
These one-dimensional factors are cut out by the action of the Hecke algebra and
the involution $w_p$ over $K$ on $J_\chi(p)$.}
The Fricke involution $w_p$ of $X_\chi(p)$ is defined over $K$ and
the curve $X_\chi(p)$ corresponds to the congruence subgroup
\begin{equation*}
\Gamma_\chi(p)=
\left\{\mtwo{a}{b}{c}{d}\in\Gamma_0(p)\,\colon\,
\chi(a)= 1\right\},
\end{equation*}
where $\chi$ denotes the Legendre character $\chi=(\frac{\cdot}{p})$ of conductor $p$.
To simplify the discussion here, we will assume that $p>5$.
Let $f$ be the modular unit on the curve $X_\chi(p)$
introduced by Ogg an Ligozat, as described by Mazur~\cite[pp.~107--108]{mazur:eisen}.
In this paper we define a certain normalization $u$ of $f$ and 
use its Fourier expansion and that of
the composite $\breve{u} = u\circ w_p$
to generalize a limit formula due Schur.
(See Proposition~\ref{prop:schur}.)
We use this limit formula together with
a monotonicity theorem of Bateman and Erd\H{o}s~\cite{erdos:bateman},
a consequence of the work of Meinardus~\cite{Mein} (described in the appendix),
and a ratio Tauberian theorem due to Sato~\cite{sato:ratio},
to prove the following.
\begin{theorem}\label{thm}
With the above assumptions,
for each $n\in\ZZ_{\geq 0}$ let $p_\pm(n)$ denote the number of partitions
\begin{equation*}
n=\lambda_1 + \dots +\lambda_r
\end{equation*}
with parts $\lambda_i\in\ZZ_{>0}$ such that $\lambda_1\leq\lambda_2\leq\dots\leq\lambda_r$ and $\chi(\lambda_i)=\pm 1$,
for $i = 1,2,\dots,r$.
Then
\begin{equation*}
\lim_{\nu\to \infty}\frac{\sum_{n=0}^\nu p_+(n)}{\sum_{n=0}^\nu p_-(n)} = \varepsilon_K^{h_K}.
\end{equation*}
\end{theorem}
Of course, the above theorem follows directly from
the classical Stolz-Ces\`aro theorem~\cite[p. 14]{polyaszego:aufgaben}
applied to a celebrated partition limit formula
due to Petersson~\cite{petersson:partitionen},
\begin{equation}\label{eqn:petersson}
\lim_{n\to\infty}\frac{p_+(n)}{p_-(n)} = \varepsilon_K^{h_K}.
\end{equation}
But our proof does not use Eq.~(\ref{eqn:petersson}).
In fact, Petersson's partition limit formula follows directly from Theorem~\ref{thm}
and a converse of the Stolz-Ces\`aro theorem\footnote{P\u{a}lt\u{a}nea stated his theorem as 
a converse of L'H\^opital rule for locally integrable functions.
But applying his theorem to suitable step functions yields a converse of the Stolz-Ces\`aro theorem.}
due to P\u{a}lt\u{a}nea~\cite{paltanea:criterion},
provided we assume a special case of
Conjecture~\ref{conj:ccb}, discussed in Section~\ref{sec:conj}.
This is a monotonicity conjecture
which looks quite natural in the context of the
monotonicity theorems of Bateman and Erd\H{o}s~\cite{erdos:bateman}.

Petersson obtained Eq.~(\ref{eqn:petersson})
by first establishing the asymptotic expression for $p_+(n)$ and for $p_-(n)$ separately,
after a rather laborious calculation.
So given the simplicity of Eq.~(\ref{eqn:petersson}),
it seems desirable to have a simpler proof.
In fact, Grosswald~\cite{grosswald:elementary} conjectured that a
monotonicity theorem of Bateman and Erd\H{o}s~\cite{erdos:bateman}
together with a suitable Tauberian converse to 
the Stolz-Ces\`aro theorem, would furnish a nicer proof of Eq.~(\ref{eqn:petersson}).
It is hoped that our approach can shed new light on Grosswald's conjecture.
Moreover,
the key role played here by the modular unit $u$
on $X_\chi(p)$ and the Fricke involution $w_p$ of $X_\chi(p)$ may help
pave the way towards an explanation
why $h_K$ and $\varepsilon_K$ appear in Eq.~(\ref{eqn:petersson}),
a question which was raised by Petersson~\cite{petersson:partitionen}.

The rest of the paper is organized as follows.
In Section~\ref{sec:fourier} we use Klein forms to define the modular unit $u$ on $X_\chi(p)$.
Then we use the class number formula for real quadratic fields
to obtain the constant term of the Fourier expansion of $u$.
We express the Fourier expansion of $\breve{u}$ as an infinite product
and conclude this section with a discussion of the $p=5$ case,
where we express the Rogers-Ramanujan continued fraction in terms of $\breve{u}$.
In Section~\ref{sec:limit} we use the Fourier expansions of $u$ and of $\breve{u}$
to obtain a generalization of a limit formula due to Schur,
which we use to prove Theorem~\ref{thm}.
In Section~\ref{sec:conj} we discuss Conjecture~\ref{conj:ccb},
including the numerical evidence that supports it,
and also suggest an open question.
In the appendix Luca shows how Eq.~(\ref{eqn:petersson})
follows from the work of Meinardus.
The appendix also includes a discussion of the inequality~(\ref{ineq}),
which is used in our proof of Theorem~\ref{thm}.

\section{Two Fourier expansions}\label{sec:fourier}
Following Kubert and Lang~\cite[p. 27]{lang:munits},
for each $z\in\CC$ and each lattice $L\subset\CC$ we may define the \textit{Klein form}
\begin{equation*}
\mathfrak{k}(z, L) = e^{-\frac{1}{2}\eta(z,L)z}\sigma(z,L),
\end{equation*}
where $\sigma(z,L)$ is the Weierstra\ss\ sigma-function and $z\mapsto\eta(z,L)$
is the $\RR$-linear function that gives the quasi-periods of the Weierstra\ss\ zeta-function
with respect to the lattice $L$.
Put $\mathfrak{k}_a(\tau) = \mathfrak{k}(z,L_\tau)$,
where the point $a=(a_1,a_2)\in\RR^2$
is uniquely determined by $z = a_1\tau + a_2$ and $L_\tau = \ZZ\tau\oplus \ZZ$,
with $\tau$ lying in the Poincar\'e upper-half plane
\begin{equation*}
\mathcal{H}=\{z\in\CC:\Im(z)>0\}.
\end{equation*}
As before, consider a prime number $p>5$ such that $p\equiv 1\pmod{4}$ and define
\begin{equation*}
u(\tau) = \prod_{r=1}^{\frac{p-1}{2}}\mathfrak{k}_{(0,r/p)}(\tau)^{\chi(r)}.
\end{equation*}
As we shall see, up to a multiplicative constant this is Ogg and Ligozat's modular unit $f$ on
the Nebentypus cover $X_\chi(p)$
described by Mazur~\cite[pp.~107--108]{mazur:eisen}.
The cover $X_\chi(p)$ has 4 cusps,
namely $\infty$ and $\overline{\infty}$ conjugate over $K$,
above the cusp $\infty$ of $X_0(p)$
and cusps $o$ and $\overline{o}$ defined over $\QQ$, above the cusp $o$ of $X_0(p)$.
Mazur also showed that
\begin{equation*}
(f) = \frac{1}{2}B_{2,\chi}( (o) - (\overline{o}) ),
\end{equation*}
where $B_{n,\chi}$ is the generalized $n$-th Bernoulli number attached to $\chi$ defined by
\begin{equation*}
\sum_{n=0}^\infty B_{n,\chi} \frac{X^n}{n!} = \sum_{r=1}^p \chi(r) \frac{X e^{rX}}{e^{pX}-1}.
\end{equation*}
The Fricke involution $w_p$ of $X_\chi(p)$ interchanges the cusps $o$ and $\infty$
(resp. $\overline{o}$ and $\overline{\infty}$).
So the composite $\breve{u} = u\circ w_p$
has a zero of order $\frac{1}{2}B_{2,\chi}$ at the cusp $\infty$ of $X_\chi(p)$.
The following proposition provides further details.
\begin{proposition}\label{prop}
We have Fourier expansions
\begin{equation*}
\breve{u}(\tau) = q^{\frac{1}{2}B_{2,\chi}}\prod_{n=1}^\infty(1-q^n)^{\chi(n)},
\end{equation*}
and
\begin{equation*}
u(\tau)=\varepsilon_K^{-h_K}(1-\sqrt{p}\,q_\tau+\dots),
\end{equation*}
where $q_\tau=e^{2\pi i\tau}$
and $\tau$ lies in the Poincar\'e upper-half plane $\mathcal{H}$.
Actually, 
\begin{equation*}
u = \varepsilon_K^{-h_K} f,
\end{equation*}
where $f$ is the modular unit of Ogg and Ligozat.
\end{proposition}
\begin{proof}
Let $\eta(\tau)$ denote Dedekind's eta-function
\begin{equation*}
\eta(\tau)=e^{\pi i\tau/12}\prod_{n=1}^\infty(1 - e^{2\pi i n\tau}).
\end{equation*}
The \textit{Siegel function} $g_a(\tau)=\mathfrak{k}_a(\tau)\eta(\tau)^2$
has a product expansion
\begin{equation*}
g_a(\tau)=-q_\tau^{\frac{1}{2}B_2(a_1)} e^{2\pi i a_2(a_1 - 1)/2} (1 - q_z)
\prod_{n=1}^\infty(1 - q_\tau^n q_z)(1 - q_\tau^n q_z^{-1}),
\end{equation*}
where $B_2(X) = X^2 - X + \frac{1}{6}$
is the second Bernoulli polynomial, and $q_z=e^{2\pi iz}$ with $z\in\CC$.   
So if we let $\zeta_p=e^{2\pi i/p}$, then
\begin{align*}
u(\tau) &= \prod_{r=1}^{\frac{p-1}{2}}g_{(0,r/p)}(\tau)^{\chi(r)} \\
        &=  \prod_{r=1}^{\frac{p-1}{2}}
\left(\zeta_p^{-\frac{r}{2}}(1 - \zeta_p^r)
\prod_{n=1}^\infty(1 - q_\tau^n\zeta_p^r)(1 - q_\tau^n\zeta_p^{-r})
\right)^{\chi(r)} \\
        &= \left(
	         \prod_{r=1}^{\frac{p-1}{2}}
                 \zeta_p^{-\chi(r)\frac{r}{2}}(1 - \zeta_p^r)^{\chi(r)}
           \right)
	   \left(
\prod_{n=1}^\infty\prod_{r=1}^{\frac{p-1}{2}}(1 - q_\tau^n\zeta_p^r)^{\chi(r)}(1 - q_\tau^n\zeta_p^{-r})^{\chi(r)}
           \right) \\
        &= \kappa f(\tau),
\end{align*}
where
\begin{equation*}
\kappa = \prod_{r=1}^{\frac{p-1}{2}}(\zeta_p^{-\frac{r}{2}} - \zeta_p^{\frac{r}{2}})^{\chi(r)}
       = \varepsilon_K^{-h_K}.
\end{equation*}
The last equality follows from the first equation in
Th\'eor\`eme~1 of Borevi\v c and \v Safarevi\v c~\cite[p.~385]{borevitch:nombres}, namely
\begin{equation*}
h_K = -\frac{1}{\log\varepsilon_K}\sum_{\substack{(r,D)=1 \\ 0<r<\frac{D}{2}}}\chi(r)\log\sin \frac{\pi r}{D},
\end{equation*}
specialized to the positive fundamental discriminant $D=p$,
which is a well-known consequence of the formula 
\begin{equation*}
L(1,\chi)=\frac{2h_K}{\sqrt{p}}\log\varepsilon_K.
\end{equation*}
Here $L(s,\chi)$ is the Dirichlet $L$-series attached to $\chi$
\begin{equation*}
L(s,\chi)=\sum_{n=1}^\infty \frac{\chi(n)}{n^s}\quad\Re(s)>0.
\end{equation*}
Therefore $u = \varepsilon_K^{-h_K} f$,
which is the third assertion in our proposition.
To prove  the second assertion note that 
\begin{equation*}
f(\tau)=\prod_{n=1}^\infty\Psi(q^n),
\end{equation*}
where
\begin{align*}
\Psi(X) &=\prod_{r=1}^{\frac{p-1}{2}}(1-X\zeta_p^r)^{\chi(r)}(1-X\zeta_p^{-r})^{\chi(r)}\\
        &=\prod_{r=1}^{\frac{p-1}{2}}(1-X\zeta_p^r)^{\chi(r)}(1-X\zeta_p^{-r})^{\chi(-r)}\\
        &=\prod_{r=1}^{p-1}(1-X\zeta_p^r)^{\chi(r)}\\
	&\equiv 1 - S_pX \pmod{X^2}.
\end{align*}
where $S_p=\sum_{r=1}^{p}\chi(r)\zeta_p^r$ is the Gau\ss\ sum attached to $\chi$.
But we assumed that $p\equiv 1\pmod{4}$, so $S_p=\sqrt{p}$.
Hence
\begin{equation*}
f(\tau)= 1 - \sqrt{p}\,q_\tau + \dots.
\end{equation*}
and the third assertion of our proposition implies that 
\begin{equation*}
u(\tau) =\varepsilon_K^{-h_K}f(\tau)= \varepsilon_K^{-h_K}(1 - \sqrt{p}\,q_\tau + \dots ).
\end{equation*}
To prove the first assertion of our proposition recall that
the Fricke involution $w_p$ of $X_\chi(p)$ is induced by
the M\"{o}bius transformation
\begin{equation*}
\tau\mapsto-\frac{1}{p\tau}
\end{equation*}
acting on
the extended upper-half plane $\mathcal{H}^* =\mathcal{H}\cup\PP^1(\QQ)$,
which is the composition of the M\"{o}bius transformation attached to
\begin{equation*}
S=\mtwo{0}{-1}{1}{0}\in\textrm{SL}_2(\ZZ)
\end{equation*}
with the map $\tau\mapsto p\tau$.
But the basic properties \textbf{K0} and \textbf{K1} of Kubert and Lang~\cite[p. 27]{lang:munits}
imply that for each $\alpha\in\textrm{SL}_2(\ZZ)$ and each $\tau\in\mathcal{H}$ we have
\begin{equation*}
\mathfrak{k}_{a\alpha}(\tau) = (c\tau + d)\mathfrak{k}_a(\alpha\tau).
\end{equation*}
Therefore
\begin{align*}
\breve{u}(\tau) &=  \prod_{r=1}^{\frac{p-1}{2}}\mathfrak{k}_{(-r/p, 0)}(p\tau)^{\chi(r)}\\
                     &=
\prod_{r=1}^{\frac{p-1}{2}}
\left(
q_\tau^{\frac{1}{2}B_2(\frac{r}{p})p}
(1-q_\tau^r)\prod_{n=1}^\infty (1 - q_\tau^{pn + r})(1 - q_\tau^{pn - r})
\right)^{\chi(r)}\\
                     &= q_\tau^{\frac{1}{2}B_{2,\chi}}
\prod_{r=1}^{\frac{p-1}{2}}(1-q_\tau^r)^{\chi(r)}
\prod_{r=1}^{\frac{p-1}{2}}
\prod_{n=1}^\infty (1-q_\tau^{pn + r})^{\chi(r)}(1-q_\tau^{pn - r})^{\chi(-r)}\\
                     &= q_\tau^{\frac{1}{2}B_{2,\chi}}\prod_{n=1}^\infty(1-q_\tau^n)^{\chi(n)},
\end{align*}
which is the first assertion of our proposition.
\end{proof}
For $p=5$ we have $\frac{1}{2}B_{2,\chi} = \frac{1}{5}$
and we may see from the above proposition that in this case $u$
is not an element of the function field of the curve $X_\chi(p)$,
but $u^5$ is in fact a Hauptmodul for $X_\chi(p)$.
Moreover, in this notation the Rogers-Ramanujan continued fraction becomes
\begin{equation}\label{eqn:rr}
\breve{u}(\tau)=\cfrac{q^{1/5}}{1+\cfrac{q}{1+\cfrac{q^2}{1+\cfrac{q^3}{1+\ddots}}}}.
\end{equation}

\section{A limit formula}\label{sec:limit}
For each $n\in\ZZ_{\geq 0}$ let $p_A(n)$ denote the number of partitions
\begin{equation*}
n=\lambda_1 + \dots +\lambda_r
\end{equation*}
with parts $\lambda_i\in A$
such that $\lambda_1\leq\lambda_2\leq\dots\leq\lambda_r$,
for $i = 1,2,\dots,r$.
In particular,
consider the set of quadratic residues $S_+$ and the set of quadratic non-residues  $S_-$
modulo $p$,
\begin{equation*}
S_\pm = \left\{m\in\ZZ_{> 0}: \chi(m) = \pm 1\right\},
\end{equation*}
so that $p_\pm(n)=p_{S_\pm}(n)$.
Here as before, $\chi=(\frac{\cdot}{p})$ is the Legendre character attached to $p$.
\begin{proposition}\label{prop:schur}
As before consider a prime $p\equiv 1\pmod{4}$.
We have the limit
\begin{equation*}
\lim_{t\to 0^+}\frac{\sum_{n=0}^\infty p_+(n)e^{-2\pi nt}}{\sum_{n=0}^\infty p_-(n)e^{-2\pi nt}} = \varepsilon_K^{h_K}.
\end{equation*}
\end{proposition}
\begin{proof}
From the first part of Proposition~\ref{prop} we have
\begin{equation*}
\frac{1}{\breve{u}} =
q^{-\frac{1}{2}B_{2,\chi}}\frac{\prod_{m\in S_+}\frac{1}{1-q^m}}{\prod_{m\in S_-}\frac{1}{1-q^m}} =
q^{-\frac{1}{2}B_{2,\chi}}\frac{\sum_{n=0}^\infty p_+(n)q^n}{\sum_{n=0}^\infty p_-(n)q^n}.
\end{equation*}
So the second part of Proposition~\ref{prop} yields
\begin{equation*}
\lim_{t\to 0^+}\breve{u}(it) =\lim_{t\to 0^+}u(w_p(it)) = \lim_{t\to\infty}u(it) = \varepsilon_K^{-h_K}
\end{equation*}
and the proposition follows.
\end{proof}
The above proposition is a generalization of a limit formula due to
Schur\cite[p.~321]{schur:1917} for $p=5$. It may be regarded as
a consequence of the Rogers-Ramanujan continued fraction, since
the right-hand side of Eq.~(\ref{eqn:rr}) tends to
\begin{equation*}
\cfrac{1}{1+\cfrac{1}{1+\cfrac{1}{1+\cfrac{1}{1+\ddots}}} }= \frac{-1+\sqrt{5}}{2}
\end{equation*}
as $t\to 0$, and we know that $\varepsilon_K^{-1}=\frac{-1+\sqrt{5}}{2}$ and that $h_K=1$
for the real quadratic field $K=\QQ(\sqrt{5})$.
\begin{remark}
From the first two terms of the Fourier expansion of $u(\tau)$
we may see that the real-analytic function $h:\RR_{> 0}\longrightarrow\RR_{> 0}$ defined by $t\mapsto \breve{u}(it)$
is a monotone concave function in a neighbourhood of $t=0$,
as depicted in Figure~\ref{cap:u} for $p=13$.
Moreover, the function $h$ is log-concave on $\RR_{> 0}$.
This is an easy consequence of the fact that the logarithmic derivative of $u(\tau)$ is,
up to a positive scalar multiple,
the well-known Eisenstein series $G_{2,\chi}(\tau)$
of weight $2$ attached to the character $\chi$.
(Cf. Lang~\cite[p. 250]{lang:modular}.)
\begin{figure}
\begin{center}
\includegraphics{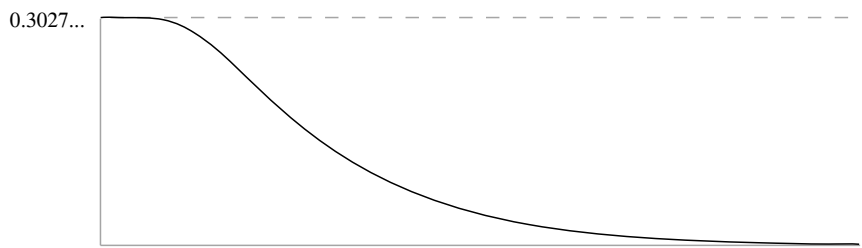}%
\end{center}
\caption{Graph of $t\mapsto\breve{u}(it)$ near $t=0$ for $p=13$.}\label{cap:u}
\end{figure}
\end{remark}
Fix $k\in\ZZ$ and define $p^{(k)}(n)=p_A^{(k)}(n)$ by the formal power series equality
\begin{equation*}
\sum_{n=0}^\infty p^{(k)}(n)X^n = (1-X)^k\prod_{a\in A}\frac{1}{1-X^a}
\end{equation*}
Following Bateman and Erd\H{o}s~\cite{erdos:bateman},
we say that a subset $A\subset\ZZ_{>0}$ satisfies property $P_k$
if $|A|>k$ and if $\textrm{gcd}(A\setminus S)=1$,
for each $S\subset A$ such that $|S|=k$.
Note that $p^{(k)}(n)$ is the $k$-th difference of $p(n)$ if $k>0$,
the $-k$-th order summatory function of $p(n)$, and $p^{(0)}(n)=p(n)$.
\begin{lemma}\label{lem:ccb}
If $A$ is an infinite subset of $\ZZ_{>0}$ such that $\textrm{gcd}(A)=1$,
then for each positive integer $h$ we have the limit
\begin{equation*}
\lim_{n\to \infty}\frac{p(0)+\dots+p(n + h)}{p(0)+\dots+p(n)} = 1.
\end{equation*}
\end{lemma}
\begin{proof}
From Bateman and Erd\H{o}s~\cite[p. 10]{erdos:bateman}, the  corollary after Theorem 6
says that for each positive integer $h$ we have
\begin{equation*}
\frac{p^{(k-1)}(n+h)-p^{(k-1)}(n)}{h} = \left(1 + o(1)\right)p^{(k)}(n).
\end{equation*}
For $k=0$ the assumption $\textrm{gcd}(A)=1$ yields
\begin{equation}\label{eqn:smallo}
\frac{p(n+1)+ \dots + p(n+h)}{p(n)} = \left(1 + o(1)\right)h,
\end{equation}
for each positive integer $h$.
But from Theorem 5 of Bateman and Erd\H{o}s~\cite[p. 7]{erdos:bateman} we know that
\begin{equation*}
\lim_{n\to \infty}\frac{p^{(k+1)}(n)}{p^{(k)}(n)} = 0,
\end{equation*}
provided $A$ is infinite and such that property $P_k$ holds.
In particular, for $k=-1$ we see that property $P_k$ is trivially satisfied
and we thus have the limit
\begin{equation*}
\lim_{n\to \infty}\frac{p(n)}{p(0) + \dots + p(n)} = 0
\end{equation*}
This limit together with Eq.~(\ref{eqn:smallo}) yield
\begin{equation*}
\lim_{n\to \infty}\frac{p(n+1)+ \dots + p(n+h)}{p(0) + \dots + p(n)} = 0,
\end{equation*}
which gives
\begin{equation*}
\lim_{n\to \infty}\frac{p(0)+\dots+p(n + h)}{p(0)+\dots+p(n)} = 
1 + \lim_{n\to \infty}\frac{p(n+1)+ \dots + p(n+h)}{p(0) + \dots + p(n)} = 1
\end{equation*}
and the proposition follows.
\end{proof}
Now we shall prove Theorem~\ref{thm}.
From the appendix, for all large enough $n$ we have
\begin{equation}\label{ineq}
p_-(n) < p_+(n).
\end{equation}
But Lemma~\ref{lem:ccb} yields
\begin{equation*}
\frac{\sum_{m=0}^\mu p_+(m)}{\sum_{n=0}^\nu p_+(n)}\to 1
\quad\textrm{as}\quad \mu, \nu\to\infty \quad\textrm{with}\quad \frac{\mu}{\nu}\to 1.
\end{equation*}
Hence the limit formula of Proposition~\ref{prop:schur}
satisfies all the hypothesis
of Theorem 3.2 of Sato~\cite[p.~85]{sato:ratio}
and Theorem~\ref{thm} follows.
\begin{remark}
If we replace Inequality~\ref{ineq} by the weaker condition
\begin{equation*}
p_-(n) = O( p_+(n)),
\end{equation*}
then Sato's ratio Tauberian theorem still applies here.
\end{remark}

\section{A conjecture and some open questions}\label{sec:conj}
Given $k\in\ZZ$, let $p^{(k)}(n)$ be as in Section~\ref{sec:limit} and
for each $n\in\ZZ_{\geq 0}$ define
\begin{equation*}
\rho^{(k)}(n)=\frac{p^{(k+1)}(n)}{p^{(k)}(n)}=\frac{p^{(k)}(n)-p^{(k)}(n-1)}{p^{(k)}(n)}.
\end{equation*}
Note that the corollary of Theorem 5 of Bateman and Erd\H{o}s~\cite[p. 9]{erdos:bateman}
says that if $A$ has property $P_k$ then
\begin{equation*}
p^{(k)}(n)\to \infty
\end{equation*}
and
\begin{equation*}
\rho^{(k)}(n)\to 0,
\end{equation*}
as $n\to\infty$.
They also show that if $A$ satisfies property $P_{k+1}$,
then $p^{(k)}(n)$ is eventually strictly increasing.
But the question of the monotonicity of $\rho^{(k)}(n)$
has not been raised before.
This is an interesting question,
as the eventual monotonicity of $\rho^{(k)}(n)$
is a natural generalization of the log-concavity of $p(n)$
for all large enough $n$.
Indeed, we have
\begin{equation*}
\rho^{(k)}(n)>\rho^{(k)}(n+1)
\end{equation*}
if and only if 
\begin{equation*}
 0< p^{(k)}(n)^2-p^{(k)}(n-1)p^{(k)}(n+1).
\end{equation*}
This monotonicity question has been settled for the case $A=\ZZ_{>0}$ and $k=0$ by
DeSalvo and Pak~\cite{desalvo:pak}, as they proved that the classical partition function $p(n)$
is log-concave for all $n>25$.
Moreover,
we may also see that
the monotonicity of $\rho^{(-1)}(n)$ for all large enough $n$ is equivalent to having
the sequence
\begin{equation}\label{eqn:pal}
\left\{\frac{1}{p(\nu)}\sum_{n=0}^\nu p(n)\right\}_{\nu=0}^\infty
\end{equation}
eventually strictly increasing.
Eq.~(\ref{eqn:pal}) is the key condition of the converse of Stolz-Ces\`aro theorem
due to P\u{a}lt\u{a}nea~\cite{paltanea:criterion} which
(together with Theorem~\ref{thm}) yields Petersson's partition limit formula.
Considering the Tauberian condition $T_2$
of the conjecture due to Grosswald~\cite[pp. 55-56]{grosswald:elementary},
we propose the following.
\begin{conjecture}\label{conj:ccb}
If $A\subset\ZZ_{>0}$ is such that $P_{k+1}, P_{k+2},\dots$, then $\rho^{(k)}(n)$
is eventually strictly decreasing.
\end{conjecture}
With the help of PARI/GP~\cite{pari:gp}
we obtained strong numerical evidence supporting Conjecture~\ref{conj:ccb}
for $p(n)=p_\pm(n)$ in the range $p\leq 1987$,
with $|k|\leq 5$ and $n\leq 10000$
and also for the classical partition function $p(n)$,
with $|k|\leq 10$ and $n\leq 100000$.

As described by Iwasawa~\cite[p. 61]{iwasawa:lectures},
there is a remarkable non-archimedean analogue of
\begin{equation*}
L(1,\chi)=-\frac{S_p}{p}\sum_{r=1}^{p-1}\chi(r)\log|1-\zeta_p^r|
\end{equation*}
known as Leopoldt's formula.
Moreover,
Siegel functions have natural rigid-analytic avatars.
It seems to be an interesting open question whether there are
analogues of Proposition~\ref{prop:schur}
(which is a generalization of a limit formula due to Schur)
and of Petersson's limit partition formula
within this realm.

\section*{Acknowledgments}

I am very grateful to Professor Luca for his generous help 
explaining me Meinardus' method and for writing the appendix.
He worked on this paper while visiting
the Max Planck Institute for Mathematics Bonn in 2019 and
the Max Planck Institute for Software Systems in Saarbr\"ucken in 2020.
Professor Luca thanks these institutions for financial help and excellent working conditions. 

I would also like to heartily thank Professor Rolen
for drawing my attention to the work of Grosswald
and for his encouraging comments.

\appendix

\section{By Florian Luca}

Here, we show how equation \eqref{eqn:petersson} follows from Meinardus' scheme \cite{Mein} (see also \cite{And}).  

Let us recall Meinardus' scheme. Let ${\mathcal A}\subseteq {\mathbb N}$ be a set of positive integers. Put
\begin{eqnarray*}
p_{\mathcal A}(n) & = & \#\{(\lambda_1,\ldots,\lambda_k): \lambda_1\ge \lambda_2\ge \cdots\ge \lambda_k\ge 1,~\lambda_1+\cdots+\lambda_k=n,\\
& & \lambda_i\in {\mathcal A}, i=1,\ldots,k\}
\end{eqnarray*}
for the number of partitions of $n$ with parts from ${\mathcal A}$. Writing $\{a_n\}_{n\ge 1}$ for the characteristic function of ${\mathcal A}$; that is, $a_n=1$ if $n\in {\mathcal A}$ and $a_n=0$ otherwise, the generating function of $p_{\mathcal A}$ is 
$$
\prod_{n\ge 1} (1-e^{-n\tau} )^{-a_n}=1+\sum_{n\ge 1} p_{\mathcal A}(n) e^{-n\tau},\quad  {\text{\rm with}}\quad {\text{\rm Re}}(\tau)>0.
$$
Meinardus, in his 1954 paper \cite{Mein}, makes the following assumptions:
\begin{itemize}
\item[(i)] Let
$$
D(s)=\sum_{n=1}^{\infty} a_n n^{-s},\quad {\text{\rm where}}\quad s=\sigma+it.
$$
Assume that $D(s)$ is convergent for $\sigma>\alpha>0$. Assume further that $D(s)$ can be analytically continued up to $\sigma=-c_0$, where $0<c_0<1$. Assume that for $\sigma\ge -c_0$, $D(s)$ is holomorphic except for 
$s=\alpha$ where it has a pole of order $1$ with residue $A$. Assume further that in this region, we have 
$$
D(s)=O(|t|^{c_1})
$$
as $t\to\infty$ for some $c_1>0$. 

\item[(ii)] For $\tau=y+2\pi i x$ with $y>0$ put
$$
g(\tau)=\sum_{n\ge 0} a_n e^{-n\tau}.
$$ 
Assume that for $|{\text{\rm arg}}( \tau)|>\pi/4$, $|x|\le 1/2$, one has
$$
{\text{\rm Re}}(g(\tau)-g(y))\le -c_2y^{-\varepsilon}
$$
for $y$ sufficiently small, where $c_2>0$ and $\epsilon>0$ are some positive real numbers. 
\end{itemize}

Under (i) and (ii), Meinardus proves that 
\begin{equation}
\label{eq:2}
p_{\mathcal A}(n)=C n^{\chi} e^{n^{\frac{\alpha}{\alpha+1}}\left(1+\frac{1}{\alpha}\right)(A\Gamma(\alpha+1)\zeta(\alpha+1))^{\frac{1}{\alpha+1}}}(1+O(n^{-\chi_1}))\quad {\text{\rm as}}\quad n\to\infty,
\end{equation}
where
\begin{eqnarray*}
C & = & e^{D'(0)} (2\pi(1+\alpha))^{-1/2}(A\Gamma(\alpha+1)\zeta(\alpha+1))^{\frac{1-2D(0)}{2(1+\alpha)}};\\
\chi & = & \frac{2D(0)-2-\alpha}{2(1+\alpha)}.
\end{eqnarray*}
He also gives some estimates for $\chi_1$ which we don't need. Well, let us apply it to our case. For us, $a_n=\chi(n)$ in the case of $p_{+}$ and $a_n=-\chi(n)$ in the case of $p_{-}$, where $\chi(n)={\displaystyle{\left(\frac{n}{p}\right)}}$ is the Legendre character modulo $p$. Putting again 
$$
L(s,\chi)=\sum_{n\ge 1} \frac{\chi(n)}{n^{s}},
$$
one sees easily that
$$
D_{+}(s)=\sum_{\substack{n\ge 1\\ \chi(n)=1}} n^{-s}=\frac{1}{2}\sum_{\substack{n\ge 1\\ p\nmid n}} \frac{1+\chi(n)}{n^s}=\frac{1}{2}\left(\zeta(s)(1-p^{-s})+L(s,\chi)\right),
$$
and similarly
$$
D_{-}(s)=\frac{1}{2}\left(\zeta(s)(1-p^{-s})-L(s,\chi)\right).
$$
So, we see that hypothesis (i) of Meinardus' scheme is fulfilled for both $D_{+}(s)$ and $D_{-}(s)$ with $\alpha=1$, $A=(1-p^{-1})$ since for $\sigma>-1/2$, 
$\zeta(s)$ is holomorphic except for a single pole at $s=1$ with residue $1$ and $L(s,\chi)$ is holomorphic. Condition (ii) is also fulfilled by standard results about vertical growth of $\zeta(s)$ and $L(s,\chi)$. Furthermore, since $\zeta(0)=-1/12$, and 
$L(0,\chi)=0$ (because $p\equiv 1\pmod 4$), we get that 
$$
D_{+}(0)=\frac{1}{2}\left((-1/12)(1-p^{-0})+L(0,\chi)\right)=0,
$$
and similarly $D_{-}(0)=0$. So, the ``main" terms of $p_{+}(n)$ and $p_{-}(n)$ in \eqref{eq:2} coincide up to the constants $C_{+}$ and $C_{-}$, that is
\begin{equation}
\label{eq:pete1}
\frac{p_{+}(n)}{p_{-}(n)}=(1+o(1))\frac{C_{+}}{C_{-}}=(1+o(1)) e^{D_{+}(0)'-D_{-}(0)'} \qquad {\text{\rm as}}\qquad n\to\infty,
\end{equation}
where $C_{+}=D_{+}'(0)$ and $C_{-}=D_{-}(0)'$.  Now
$$
D_{+}(s)'=\frac{1}{2}\left(\zeta'(s)(1-p^{-s})+\zeta(s)(\log p) p^{-s}+L'(s,\chi)\right).
$$
Evaluating in $s=0$, we get 
$$
D_{+}(0)'=\frac{1}{2}\zeta(0)\log p+\frac{1}{2}L'(0,\chi)=-\frac{\log p}{24}+\frac{1}{2}L(0,\chi)'.
$$
A similar argument shows that 
$$
D_{-}(0)'=-\frac{\log p}{24}-\frac{1}{2}L(0,\chi)',
$$ 
so 
$$
D_{+}(0)'-D_{-}(0)'=L(0,\chi)'.
$$
Since $L(0,\chi)'=({\sqrt{p}}/2)L(1,\chi)$, it follows that 
$$
D_{+}(0)'-D_{-}(0)'=({\sqrt{p}}/2)L(1,\chi)
$$ 
which is positive (by the proof of Dirichlet's theorem on primes in progressions). 
In fact, by the class number formula the above difference is $h_k\log \varepsilon_K$ and we recover Petersson's limit
from \eqref{eq:pete1}.

In particular, the inequality 
$$
p_+(n)>p_-(n)\quad {\text{\rm holds~for~all}}\quad n>n_0(p).
$$
One may wonder if the fact that the inequality $p_+(n)>p_-(n)$ holds might be due to the fact that $1$ is a quadratic residue and being the smallest positive integer it likely contributes to a lot of elements counted by $p_+(n)$. Well, let us test it. Let $p_{1,+}(n)$ be the number of partitions of 
$n$ with parts that are $>1$ but quadratic residues modulo $p$. Then  
$$
D_{1,+}(s)=D_{+}(s)-1,
$$
so $D_{+,1}(0)=D_{+}(0)-1=-1$. It thus follows that
$$
\chi_{1,+}=\frac{2D_{1,+}(0)-3}{4}=-\frac{5}{4}\qquad {\text{\rm while}}\qquad \chi_{-}=\frac{2D_{-}(0)-3}{4}=-\frac{3}{4},
$$
therefore
$$
\frac{p_{1,+}(n)}{p_{-}(n)}=(1+o(1)) c_3 n^{-1/2}\qquad {\text{\rm as}}\qquad n\to\infty,
$$
where $c_3:=C_{1,+}/C_{-}$. The above asymptotic shows that the inequality $p_{1,+}(n)<p_{-}(n)$ holds for large $n$. So, indeed, if we eliminate the $1$'s from  the partitions of $p_{+}(n)$ we get a number of partitions much smaller (in fact, of a smaller order of magnitude asymptotically) than $p_{-}(n)$, whereas $p_{+}(n)$ and $p_{-}(n)$ are of the same order or magnitude, which can be indeed interpreted by saying  that the fact that the inequality $p_{+}(n)>p_{-}(n)$ holds for large $n$ is driven by the contribution of the $1$'s in the $p_+(n)$ side.


\begin{thebibliography}{0}

\bibitem{And}
G.\thinspace{}E. Andrews, \emph{The theory of partitions}, Encyclopedia of
  Mathematics and its Applications, vol.~2, Cambridge University Press, 1998,
  Reprint of the 1976 original.

\bibitem{erdos:bateman}
P.~T. Bateman and P.~Erd\H{o}s, \emph{Monotonicity of partition functions},
  Mathematika \textbf{3} (1956), no.~1, 1--14.

\bibitem{borevitch:nombres}
Z.\thinspace{}I. {Borevi\v c} and I.\thinspace{}R. {\v Safarevi\v c},
  \emph{Th\'eorie des nombres}, Monographies internationales de math\'ematiques
  modernes, Gauthier-Villars, Paris, 1967.

\bibitem{desalvo:pak}
S.~DeSalvo and I.~Pak, \emph{Log-concavity of the partition function}, The
  Ramanujan Journal \textbf{38} (2015), no.~1, 61--73.

\bibitem{grosswald:elementary}
E.~Grosswald, \emph{Elementary proofs in the theory of partitions},
  Mathematische Zeitschrift \textbf{81} (1963), no.~1, 52--61.

\bibitem{iwasawa:lectures}
K.~Iwasawa, \emph{Lectures on $p$-adic {L}-functions}, Annals of Mathematics
  Studies, vol.~74, Princeton University Press, Princeton, NJ, 2009.

\bibitem{lang:munits}
D.\thinspace{}S. Kubert and S.~Lang, \emph{Modular units}, Grundlehren der
  Mathematischen Wissenschaften, vol. 244, Springer-Verlag, New York, 1981.

\bibitem{lang:modular}
S.~Lang, \emph{Introduction to modular forms}, Grundlehren der Mathematischen
  Wissenschaften, vol. 222, Springer-Verlag, Berlin, 1987.

\bibitem{mazur:eisen}
B.~Mazur, \emph{Modular curves and the {E}isenstein ideal}, Inst. Hautes
  \'Etudes Sci. Publ. Math. \textbf{47} (1977), 33--186.

\bibitem{Mein}
G.~Meinardus, \emph{Asymptotische {A}ussagen \"uber {P}artitionen},
  Mathematische Zeitschrift \textbf{59} (1954), no.~1, 388--398.

\bibitem{petersson:partitionen}
H.~Petersson, \emph{{\"Uber} modulfunktionen und partitionenprobleme},
  Abhandlungen der Deutschen Akademie der Wissenschaften zu Berlin,
  Akademie-Verlag, Berlin, 1954, Klasse f\"ur Mathematik und Allgemeine
  Naturwissenschaften, pp.~7--59.

\bibitem{polyaszego:aufgaben}
G.~P\'olya and G.~Szeg\H{o}, \emph{Aufgaben und {L}ehrs\"atze aus der
  {A}nalysis, {I}}, Grundlehren der Mathematischen Wissenschaften, vol.~19,
  Springer-Verlag, Berlin, 1925.

\bibitem{paltanea:criterion}
E.~P\u{a}lt\u{a}nea, \emph{A criterion for the limit of a ratio of functions},
  Bull. Transilv. Univ. Bra\cb{s}ov Ser. {III} \textbf{10}(59) (2017), 115--120.

\bibitem{sato:ratio}
R.~Sato, \emph{Ratio {T}auberian theorems for relatively bounded functions and
  sequences in banach spaces}, Comment. Math. Univ. Carolin. \textbf{52}
  (2011), no.~1, 77--88.

\bibitem{schur:1917}
I.~Schur, \emph{Ein {B}eitrag zur additiven {Z}ahlentheorie und zur {T}heorie der {K}ettenbr\"uche},
  Sitzungsberichte der K\"oniglich Preussischen Akademie der Wissenschaften, Verlag der K\"{o}niglichen
  Akademie der Wissenschaften, Berlin, 1917, pp.~302--321.

\bibitem{shimura:lnm}
G.~Shimura, \emph{Class fields over real quadratic fields in the theory of
  modular functions}, Several Complex Variables {II}, Lecture Notes in
  Mathematics, vol. 185, Springer, Berlin, Heidelberg, 1971, pp.~169--188.
  
\bibitem{pari:gp}
The PARI~Group, Univ. Bordeaux, \emph{{PARI/GP}, {V}ersion \texttt{2.11.1}},
  2020, \\\protect{\sf http://pari.math.u-bordeaux.fr/}.

\end{thebibliography}
\end{document}